\documentclass[11pt]{article}

\usepackage[pagebackref,colorlinks=true,pdfpagemode=none,urlcolor=blue,
linkcolor=blue,citecolor=blue]{hyperref}

\usepackage{amsmath,amsfonts,amssymb,amsthm}
\usepackage{mathrsfs}
\usepackage{color}

\newtheorem{theorem}{Theorem}[section]
\newtheorem{lemma}[theorem]{Lemma}

\newtheorem{proposition}[theorem]{Proposition}

\theoremstyle{remark}
\newtheorem{remark}[theorem]{Remark}

\renewenvironment{proof}[1][Proof]{ {\itshape \noindent {#1.}} }{$\Box$
\medskip}

\numberwithin{equation}{section}
\newcommand{\R}{\mathbb{R}}

\newcommand{\Z}{\mathbb{Z}}
\newcommand{\Pb}{\mathbb{P}}
\newcommand{\E}{\mathbb{E}}

\newcommand{\G}{\mathcal{G}}

\newcommand{\C}{\mathcal{C}}

\newcommand{\V}{\mathbb{V}}

\def\les{\lesssim}

\begin{document}
\title{An Invariance Principle for Brownian Motion in Random Scenery}
\author{Yu Gu\thanks{Department of Applied Physics \& Applied
Mathematics, Columbia University, New York, NY 10027 (yg2254@columbia.edu; gb2030@columbia.edu)}  \and Guillaume Bal\footnotemark[1]}
\maketitle

\begin{abstract}
We prove an invariance principle for Brownian motion in Gaussian or Poissonian random scenery by the method of characteristic functions. Annealed asymptotic limits are derived in all dimensions, with a focus on the case of dimension $d=2$, which is the main new contribution of the paper.
\end{abstract}

\section{Introduction}

In this paper, we study the asymptotic distributions of random processes of the form $\int_0^{t}V(B_s)ds$, with $V$ some stationary random potential and $B_s, s\in [0,1]$ a standard Brownian motion independent of $V$.

The corresponding discrete version is the Kesten-Spitzer model of random walk in random scenery \cite{kesten1979limit} of the form $W_n=\sum_{i=1}^n\xi_{S_k}$. Here, $S_k=X_1+\ldots+X_k$ is a random walk on $\Z$ with i.i.d. increments and $\xi_n, n\in \Z$, are i.i.d. and independent of $X_i$. When $X_i$ and $\xi_i$ belong to the domain of attraction of certain stable laws, then after proper scaling $a(n)^{-1}W_{[nt]}$ converges weakly as $n\to \infty $ to a self-similar process with stationary increment. Non-stable limits may appear in that case. Assuming moreover that $\xi_i$ has zero mean and finite variance, it is shown in \cite{bolthausen1989central} that $(n\log n)^{-\frac{1}{2}}W_{[nt]}$ converges weakly to a Brownian motion when $d=2$. When $d\geq 3$, the argument contained in \cite{kesten1979limit} essentially proves that $n^{-\frac{1}{2}}W_{[nt]}$ converges weakly to a Brownian motion.

The continuous version $a(n)^{-1}\int_0^{nt}V(B_s)ds$ has been analyzed in \cite{remillard1991limit} for piecewise constant potentials given by $V(x)=\xi_{[x+U]}$, where $\xi_i$ are i.i.d. random variables with zero mean and finite variance, and $U$ is uniformly distributed in $[0,1)^d$ and independent of $\xi_i$. The results are similar to those obtained in the discrete setting. In \cite{kipnis1986central}, Kipnis-Varadhan proved central limit results in both the discrete and continuous settings for additive functionals of Markov processes. For the special case of Brownian motion in random scenery and by adapting the point of view of "medium seen from an observer", their results can be applied to prove invariance principle for the most general class of $V(x)$ when $d\geq 3$, including the ones analyzed in \cite{remillard1991limit}. For more relevant results and backgrounds, see \cite{komorowski2012fluctuations} and the references therein.

In this paper, we consider two types of simple yet important potentials, namely the Gaussian and Poissonian potentials, and derive the asymptotic distributions of $a(n)^{-1}\int_0^{nt}V(B_s)ds$ in all dimensions by method of characteristic functions. Since \cite{pardoux2006homogenization} contains the results for $d=1$ while \cite{kipnis1986central} implies the results for $d\geq 3$, our main contribution is the case $d=2$. For Gaussian and Poissonian potentials, the method of characteristic functions offers a relatively simple proof, which we present  in all dimensions.

There are several physical motivations for studying functionals of the form $\int_0^tV(B_s)ds$. We mention two examples here. The first is the parabolic Anderson model $u_t=\frac{1}{2}\Delta u+V u$ with random potential $V$ and initial condition $f$. By Feynman-Kac formula, the solution can be written as $u(t,x)=\E_B^x\{f(B_t)\exp(\int_0^tV(B_s)ds)\}$, where $\E_B^x$ denotes the expectation with respect to the Brownian motion starting from $x$. It is clear that the large time behavior of $u(t,x)$ is affected by the asymptotics of Brownian functional $\exp(\int_0^tV(B_s)ds)$, see e.g. the applications in the context of homogenization \cite{lejay2001homogenization, pardoux2006homogenization, gu2013weak}. As a second example, if we look at the model of Brownian particle in Poissonian obstacle denoted by $V$, then the integral $\int_0^t V(B_s)ds$ measures the total trapping energy received by the particle up to time $t$ and $\exp(-\int_0^t V(B_s)ds)$ is used to define the Gibbs measure. For Brownian motion in Poissonian potential, many existing results are of large deviation type; see \cite{sznitman1998brownian} for a review of such results.

The rest of the paper is organized as follows. We first describe the assumptions on the potentials and state our main theorems in section \ref{sec:2}. We then prove the convergence of finite dimensional distributions and tightness results in section \ref{sec:3} for the non-degenerate case and section \ref{sec:4} for the degenerate case (when the power spectrum of the potential vanishes at the origin). We discuss possible applications and extensions of our results in section \ref{sec:5} and present some technical lemmas in an appendix.

Here are notations used throughout the paper. We write $a\les b$ when there exists a constant $C$ independent of $n$ such that $a\leq Cb$. $N(\mu,\sigma^2)$ denotes the normal random variable with mean $\mu$ and variance $\sigma^2$ and $q_t(x)$ is the density function of $N(0,t)$. We use $a\wedge b=\min(a,b)$ and $a\vee b=\max(a,b)$. For multidimensional integrations, $\prod_idx_i$ is abbreviated as $dx$.

\section{Problem setup and main results}
\label{sec:2}
The Gaussian and Poissonian potentials are denoted by $V_g(x)$ and $V_p(x)$, respectively, throughout the paper.

For the Gaussian case, we assume $V_g(x)$ is stationary with zero mean and the covariance function $R_g(x)=\E\{V_g(x+y)V_g(y)\}$ is continuous and compactly supported. The power spectrum $\hat{R}_g(\xi)=\int_{\R^d}R_g(x)e^{-i\xi x}dx$, and by Bochner's theorem $\hat{R}_g(0)=\int_{\R^d}R_g(x)dx\geq 0$.

For the Poissonian case, we assume
\begin{equation}
V_p(x)=\int_{\R^d} \phi(x-y)\omega(dy)-c_p,
\end{equation}
where $\omega(dx)$ is a Poissonian field in $\R^d$ with the $d-$dimensional Lebesgue measure as its intensity measure and $\phi$ is a continuous, compactly supported shape function such that $\int_{\R^d} \phi(x)dx=c_p$. It is straightforward to check that $V_p(x)$ is stationary and has zero mean, and its covariance function \begin{equation}
R_p(x)=\E\{V_p(x+y)V_p(y)\}=\int_{\R^d}\phi(x+z)\phi(z)dz
\end{equation}
is continuous and compactly supported as well. The power spectrum $\hat{R}_p(\xi)=\int_{\R^d}R_p(x)e^{-i\xi x}dx$ and since $\int_{\R^d}\phi(x)dx=c_p$, we have $\hat{R}_p(0)=c_p^2\geq 0$.

In the Poissonian case, the random field $V_p(x)$ is mixing in the following sense. For two Borel sets $A,B\subset \R^d$, let $\mathcal{F}_A$ and $\mathcal{F}_B$ denote the sub-$\sigma$ algebras generated by the field $V_p(x)$ for $x\in A$ and $x\in B$, respectively. Then there exists a positive and decreasing function $\varphi(r)$ such that
\begin{equation}
|Cor(\eta,\zeta)|\leq \varphi(2d(A,B))
\label{eq:mixingCondition}
\end{equation}
for all square integrable random variables $\eta$ and $\zeta$ that are $\mathcal{F}_A$ and $\mathcal{F}_B$ measurable, respectively. The multiplicative factor $2$ is only here for convenience. Actually, when $|x|$ is sufficiently large, $V_p(x+y)$ is independent of $V_p(y)$ and so the mixing coefficient $\varphi(r)$ can be chosen as a positive, decreasing function with compact support in $[0,\infty)$. We will use this in the estimation of the fourth moment of $V_p(x)$.

The following theorems are our main results.

\begin{theorem}
Let $B_t, t\geq 0$ be a $d-$dimensional standard Brownian motion independent of the stationary random potential $V(x)$, which is chosen to be either Gaussian or Poissonian, and $R(x)=\E\{V(x+y)V(y)\}$ be the covariance function, $\hat{R}(0)=\int_{\R^d}R(x)dx$. Define $X_n(t)=a(n)^{-1}\int_0^{nt} V(B_s)ds$ with the scaling factor
\begin{equation*}
a(n)=\left\{\begin{array}{ll}
n^{\frac{3}{4}} & d=1,\\
(n\log n)^{\frac{1}{2}}  & d=2,\\
n^{\frac{1}{2}} & d\geq3.
\end{array}\right.
\end{equation*}
Then we have that $X_n(t)$ converges weakly in $\C([0,1])$ to $\sigma_dZ_t$ with the following representations:

When $d=1$, then $Z_t=\int_{\R} L_t(x)W(dx)$, where $L_t(x)$ is the local time of $B_t$ and $W(dx)$ is a $1-$dimensional white noise independent of $L_t(x)$; $\,\,\sigma_d=\sqrt{\hat{R}(0)}$.

When $d=2$, then $Z_t$ is a standard Brownian motion; $\,\,\sigma_d=\sqrt{\hat{R}(0)/\pi}$.

When $d\geq 3$, then $Z_t$ is a standard Brownian motion; $\,\,\sigma_d=\sqrt{\pi^{-\frac{d}{2}}\Gamma(\frac{d}{2}-1)\int_{\R^d}R(x)|x|^{2-d}dx}$.
    \label{thm:mainTH}
\end{theorem}

We note that when $d\geq 3$, then $\sigma_d=\sqrt{4(2\pi)^{-d}\int_{\R^d}\hat{R}(\xi)|\xi|^{-2}d\xi}$. Since $\hat{R}(\xi)\geq0$, $\sigma_d>0$ in both cases and the limit is nontrivial. When $d\leq 2$, the limit is nontrivial if $\hat{R}(0)\neq 0$. In the degenerate case where $\hat{R}(0)=0$, for instance if $c_p=0$ in the Poissonian case in $d=1,2$, then the limit obtained in the previous theorem is trivial. The scaling factor $a(n)$ should be chosen smaller to obtain a nontrivial limit. We prove the following result:
\begin{theorem}
In $d=1,2$, let $B_t, t\geq 0$ be a $d-$dimensional standard Brownian motion independent of the stationary random potential $V(x)$, which is chosen to be either Gaussian or Poissonian, and $R(x)=\E\{V(x+y)V(y)\}$ with $\hat{R}(\xi)|\xi|^{-2}$ integrable. Define $X_n(t)=n^{-\frac12}\int_0^{nt} V(B_s)ds$. Then we have $X_n(t)$ converges weakly in $\C([0,1])$ to $\sigma W_t$, with $W_t$ a standard Brownian motion and $\sigma=\sqrt{4(2\pi)^{-d}\int_{\R^d}\hat{R}(\xi)|\xi|^{-2}d\xi}$.
    \label{thm:degenerated12}
\end{theorem}

\begin{remark}
In the degenerate case, the scaling factor $n^{-\frac12}$ is the same as in $d\geq 3$. In $d=1$, the limiting processes are different for the non-degenerate and degenerate cases. If $\hat{R}(\xi)\in L^1$, then for $\hat{R}(\xi)|\xi|^{-2}$ to be integrable,  we only need to assume that $\hat{R}(\xi)\les |\xi|^\alpha$ at the origin with $\alpha>1$ when $d=1$ and $\alpha>0$ when $d=2$.
\end{remark}

We will refer to Theorem \ref{thm:mainTH} and \ref{thm:degenerated12} as non-degenerate and degenerate cases respectively in the following sections. The proof contains convergence of finite dimensional distributions and tightness result.

\section{Non-degenerate case when $d\geq 1$}
\label{sec:3}

\subsection{Convergence of finite dimensional distributions}

We first prove the weak convergence of finite dimensional distributions through the estimation of characteristic functions.

For any $0=t_0<t_1<\ldots<t_N\leq 1$ and $\alpha_i\in \R, i=1,\ldots,N$, by considering $Y_N:=\sum_{i=1}^N \alpha_i (X_n(t_i)-X_n(t_{i-1}))$, we have the following explicit expressions.

In the Gaussian case, since $Y_N=\sum_{i=1}^N\alpha_ia(n)^{-1}\int_{nt_{i-1}}^{nt_i}V_g(B_s)ds$, we obtain
\begin{equation}
\begin{aligned}
\E\{\exp(i\theta Y_N)\}=&\E\{\E\{\exp(i\theta Y_N)|B\}\}\\
=&\E\{\exp(-\frac{1}{2}\theta^2\sum_{i,j=1}^N\alpha_i\alpha_j\frac{1}{a(n)^2}\int_{nt_{i-1}}^{nt_i}\int_{nt_{j-1}}^{nt_j}R_g(B_s-B_u)dsdu)\}.
\label{eq:chrGauss}
\end{aligned}
\end{equation}
Since $\E\{\exp(i\theta Y_N)|B\}$ is bounded by $1$, to prove convergence of $\E\{\exp(i\theta Y_N)\}$, we only need to prove the convergence in probability of $a(n)^{-2}\int_{nt_{i-1}}^{nt_i}\int_{nt_{j-1}}^{nt_j}R_g(B_s-B_u)dsdu$.

In the Poissonian case, we write
\begin{equation*}
Y_N=\int_{\R^d}\left(\sum_{i=1}^N\alpha_i \frac{1}{a(n)}\int_{nt_{i-1}}^{nt_i}\phi(B_s-y)ds\right)\omega(dy)
-\frac{c_p}{a(n)}\sum_{i=1}^N\alpha_i(nt_i-nt_{i-1}),
\end{equation*}
and straightforward calculations lead to
\begin{equation*}
\E\{ \exp(i\theta Y_N)\}=\E\{ \exp(\int_{\R^d}( e^{i\theta F_n(y)}-1)dy)\}\exp(-i\theta\frac{c_p}{a(n)}\sum_{i=1}^N\alpha_i(nt_i-nt_{i-1})),
\end{equation*}
where $F_n(y):=\sum_{i=1}^N\alpha_i a(n)^{-1}\int_{nt_{i-1}}^{nt_i}\phi(B_s-y)ds$. Since $\int_{\R^d}\phi(x)dx=c_p$, we obtain
\begin{equation}
\E\{\exp(i\theta Y_N)\}=\E\{\E\{\exp(i\theta Y_N)|B\}\}=\E\{\exp(\int_{\R^d}\sum_{k=2}^\infty \frac{1}{k!}(i\theta F_n(y))^kdy)\}.
\label{eq:chrPoisson}
\end{equation}
Similarly, $\E\{\exp(i\theta Y_N)|B\}$ is bounded by $1$, so it suffices to show the convergence in probability of $\int_{\R^d}\sum_{k=2}^\infty \frac{1}{k!}(i\theta F_n(y))^kdy$. When $k=2$,
\begin{equation}
\int_{\R^d}F_n(y)^2dy=\sum_{i,j=1}^N\alpha_i\alpha_ja(n)^{-2}\int_{nt_{i-1}}^{nt_i}\int_{nt_{j-1}}^{nt_j}R_p(B_s-B_u)dsdu
\end{equation}
is the conditional variance of $Y_N$ given $B_s$. We will see that the proof of the Poissonian case implies the Gaussian case.

\subsubsection{Poissonian case $d=1$}

When $d=1$, since $V_p$ is mixing, \cite[Theorem 2]{pardoux2006homogenization} implies the result. We give a different proof using characteristic functions.

First of all, by scaling property of Brownian motion, we do not distinguish between $F_n(y)=\sum_{i=1}^N \alpha_in^{-\frac{3}{4}}\int_{nt_{i-1}}^{nt_i}\phi(B_s-y)ds$ and $F_n(y)=\sum_{i=1}^N\alpha_in^{\frac{1}{4}}\int_{t_{i-1}}^{t_i} \phi(\sqrt{n}B_s-y)ds$. Using the second representation, we have
\begin{equation*}
F_n(y)=\sum_{i=1}^N\alpha_in^{\frac{1}{4}}\int_{t_{i-1}}^{t_i} \phi(\sqrt{n}B_s-y)ds=\sum_{i=1}^N\alpha_in^{\frac{1}{4}}\int_{\R}\phi(\sqrt{n}x-y)(L_{t_i}(x)-L_{t_{i-1}}(x))dx,
\end{equation*}
where $L_t(x)$ is the local time of $B_s$. So
\begin{equation}
\int_{\R^d}F_n(y)^2dy=\sum_{i,j=1}^N\alpha_i\alpha_j
\sqrt{n}\int_{\R^2}R_p(\sqrt{n}(z-x))(L_{t_i}(x)-L_{t_{i-1}}(x))(L_{t_j}(z)-L_{t_{j-1}}(z))dxdz.
\end{equation}

The following two propositions prove the convergence in probability of $\int_{\R^d}\sum_{k=2}^\infty \frac{1}{k!}(i\theta F_n(y))^kdy$.

\begin{proposition}
$\int_{\R}F_n(y)^2dy\to \hat{R}_p(0)\sum_{i,j=1}^N \alpha_i\alpha_j\int_{\R}(L_{t_i}(x)-L_{t_{i-1}}(x))(L_{t_j}(x)-L_{t_{j-1}}(x))dx$ almost surely.
\label{prop:conVar1d}
\end{proposition}

\begin{proposition}
$\int_{\R} \sum_{k=3}^\infty \frac{1}{k!}(i\theta F_n(y))^k dy\to 0$ almost surely.
\label{prop:conRemain1d}
\end{proposition}

\begin{proof}[Proof of Proposition \ref{prop:conVar1d}]
For any $i,j=1,\ldots,N$, we consider
\begin{equation*}
\begin{aligned}
(i):=&\sqrt{n}\int_{\R^2}R_p(\sqrt{n}(z-x))(L_{t_i}(x)-L_{t_{i-1}}(x))(L_{t_j}(z)-L_{t_{j-1}}(z))dxdz\\
=&\int_{\R^2}R_p(y)(L_{t_i}(x)-L_{t_{i-1}}(x))(L_{t_j}(x+\frac{y}{\sqrt{n}})-L_{t_{j-1}}(x+\frac{y}{\sqrt{n}}))dydx,
\end{aligned}
\end{equation*}
and because $L_t(x)$ is continuous with compact support almost surely, we have
\begin{equation*}
(i)\to \hat{R}_p(0)\int_{\R}(L_{t_i}(x)-L_{t_{i-1}}(x))(L_{t_j}(x)-L_{t_{j-1}}(x))dx,
\end{equation*}
which completes the proof.
\end{proof}

\begin{proof}[Proof of Proposition \ref{prop:conRemain1d}]
For a fixed realization of $B_s$, we have $|F_n(y)|\les n^{\frac{1}{4}}\int_{|x|<M}|\phi(\sqrt{n}x-y)|dx$, where $M$ is a constant depending on the realization and thus
\begin{equation*}
\begin{aligned}
&|\int_{\R} F_n(y)^kdy|\\
\les &n^{\frac{k}{4}}\int_{\R}\int_{[-M,M]^k}\prod_{i=1}^k |\phi(\sqrt{n}x_i-y)|dxdy=
n^{\frac{k}{4}}\int_{\R}\int_{[-M,M]^k}\prod_{i=1}^k |\phi(\sqrt{n}(x_i-x_1)+y)|dxdy\\
\leq &\frac{1}{n^{\frac{k-2}{4}}}\int_{\R^k}\int_{[-M,M]}|\phi(y)|\prod_{i=2}^k |\phi(x_i-\sqrt{n}x_1+y)|dxdy\les \frac{1}{n^{\frac{k-2}{4}}}.
\end{aligned}
\end{equation*}
Since $\sum_{k=3}^\infty \frac{1}{k!}\frac{|\theta|^k}{n^{\frac{k-2}{4}}}\to 0$ as $n\to \infty$, the proof is complete.
\end{proof}

Recalling \eqref{eq:chrPoisson}, by Proposition \ref{prop:conVar1d} and \ref{prop:conRemain1d}, we have proved the almost sure convergence of  the exponents. Therefore, by the Lebesgue dominated convergence theorem we have
\begin{equation}
\begin{aligned}
\E\{ \exp(i\theta Y_N)\}=&\E\{ \exp(\int_{\R^d}\sum_{k=2}^\infty \frac{1}{k!}(i\theta F_n(y))^kdy)\}\\
\to& \E\{\exp(-\frac{1}{2}\theta^2\hat{R}_p(0)\sum_{i,j=1}^N \alpha_i\alpha_j\int_{\R}(L_{t_i}(x)-L_{t_{i-1}}(x))(L_{t_j}(x)-L_{t_{j-1}}(x))dx)\}\\
=&\E\{\exp(i\theta\sigma_d\sum_{i=1}^N\alpha_i(Z_{t_i}-Z_{t_{i-1}}))\}
\end{aligned}
\end{equation}
when $Z_t=\int_{\R}L_t(x)W(dx)$.

\subsubsection{Poissonian case $d\geq 2$}

When $d\geq 2$, the local time does not exist, and to prove the convergence of the conditional variance of $X_n(t)$ given $B_s$, we need to calculate fourth moments. First, we define
\begin{equation}
\V_n=\E\{X_n(t)^2|B_s, s\in [0,t]\}=\frac{1}{a(n)^2}\int_0^{nt}\int_0^{nt}R_p(B_s-B_u)dsdu
\end{equation}
so that $\E\{X_n(t)^2\}=\E\{\V_n\}$. The following two lemmas show that the conditional variance converges in probability.

\begin{lemma}
$\E\{\V_n\}\to \sigma_d^2t$ as $n\to \infty$.
\label{lem:con1stmomentVar}
\end{lemma}

\begin{lemma}
$\E\{\V_n^2\} \to \sigma_d^4t^2$ as $n\to \infty$.
\label{lem:con2ndmomentVar}
\end{lemma}


In the proofs, we deal with $d=2$ and $d\geq 3$ in different ways. For the latter, we only use the fact that $\hat{R}_p(\xi)|\xi|^{-2}$ is integrable and so the proof also applies in the degenerate case. Both $R_p(x)$ and $\hat{R}_p(\xi)$ are even functions, a fact that we will use frequently in the proof.

\begin{proof}[Proof of Lemma \ref{lem:con1stmomentVar}]
We first consider the case $d=2$. For fixed $x$, by change of variables $\lambda=\frac{|x|^2}{2u}$, we have
\begin{equation}
\begin{aligned}
\E\{\V_n\}=&\frac{2}{a(n)^2}\int_0^{nt}\int_0^{s}\int_{\R^d} R_p(x)\frac{1}{(2\pi u)^{\frac{d}{2}}}e^{-\frac{|x|^2}{2u}}dxduds\\
=&\frac{n}{a(n)^2}\int_0^{t}\int_{\R^d} \int_{\frac{|x|^2}{2ns}}^\infty  R_p(x)\frac{1}{\pi^{\frac{d}{2}}}\lambda^{\frac{d}{2}-2}e^{-\lambda}\frac{1}{|x|^{d-2}}d\lambda dxds.
\end{aligned}
\end{equation}

Since $a(n)=(n\log n)^{\frac{1}{2}}$, by integrations by parts in $\lambda$, we have
\begin{equation}
\begin{aligned}
\E\{\V_n\}=&\frac{1}{\log n}\int_0^t\int_{\R^d}\frac1\pi R_p(x)\left(e^{-\frac{|x|^2}{2ns}}\log \frac{2ns}{|x|^2}+\int_{\frac{|x|^2}{2ns}}^\infty e^{-\lambda}\log \lambda d\lambda\right)dxds
\to \frac{t}{\pi}\hat{R}_p(0)
\end{aligned}
\label{eq:variance2d}
\end{equation}
by the Lebesgue dominated convergence theorem.

Consider now the case $d\geq 3$. Then, $a(n)=n^{\frac{1}{2}}$ and by Fourier transform, we have
\begin{equation}
\begin{aligned}
\E\{\V_n\}=&\frac{1}{(2\pi)^dn}\int_{[0,nt]^2}\int_{\R^d}\hat{R}_p(\xi)e^{-\frac{1}{2}|\xi|^2|s-u|}d\xi dsdu\\
=&\frac{4}{(2\pi)^d}\int_{\R^d}\frac{\hat{R}_p(\xi)}{|\xi|^2}\int_0^{t}(1-e^{-\frac{1}{2}|\xi|^2ns})dsd\xi\to \frac{4t}{(2\pi)^d}\int_{\R^d}\frac{\hat{R}_p(\xi)}{|\xi|^2}d\xi
\end{aligned}
\end{equation}
as $n\to \infty$.
\end{proof}

\begin{proof}[Proof of Lemma \ref{lem:con2ndmomentVar}]
By symmetry of $R(x)$, we write $$\V_n^2=a(n)^{-4}\int_{[0,nt]^4}R_p(B_{s_1}-B_{s_2})R_p(B_{s_3}-B_{s_4})ds=8((i)+(ii)+(iii)),$$ where
\begin{eqnarray}
(i)=\frac{1}{a(n)^4}\int_{0<s_1<s_2<s_3<s_4<nt}R_p(B_{s_1}-B_{s_2})R_p(B_{s_3}-B_{s_4})ds,\\
(ii)=\frac{1}{a(n)^4}\int_{0<s_1<s_3<s_2<s_4<nt}R_p(B_{s_1}-B_{s_2})R_p(B_{s_3}-B_{s_4})ds,\\
(iii)=\frac{1}{a(n)^4}\int_{0<s_1<s_3<s_4<s_2<nt}R_p(B_{s_1}-B_{s_2})R_p(B_{s_3}-B_{s_4})ds.
\end{eqnarray}
We consider first the case $d=2$.

$(i)$: for fixed $x,y$, by change of variables $u_1=\frac{s_1}{n},u_3=\frac{s_3-s_2}{n}, \lambda_2=\frac{|x|^2}{2(s_2-s_1)}, \lambda_4=\frac{|y|^2}{2(s_4-s_3)}$, we have
\begin{equation*}
\begin{aligned}
\E\{ (i)\}=\frac{n^2}{a(n)^4}\int_{\R_+^4}\int_{\R^{2d}}&\frac{R_p(x)}{|x|^{d-2}}\frac{R_p(y)}{|y|^{d-2}}\frac{1}{4\pi^d}\lambda_2^{\frac{d}{2}-2}e^{-\lambda_2}
\lambda_4^{\frac{d}{2}-2}e^{-\lambda_4}\\
&1_{0\leq u_1+u_3\leq t; \frac{|x|^2}{2\lambda_2}+\frac{|y|^2}{2\lambda_4}\leq n(t-u_1-u_3)}dxdyd\lambda_2d\lambda_4du_1du_3.
\end{aligned}
\end{equation*}

We define
 \begin{equation*}
 f(c)=\frac{1}{(\log n)^2}\int_{\R_+^2}\frac{1}{\lambda_2}e^{-\lambda_2}
\frac{1}{\lambda_4}e^{-\lambda_4}1_{\frac{|x|^2}{2\lambda_2}\leq cn(t-u_1-u_3);\frac{|y|^2}{2\lambda_4}\leq cn(t-u_1-u_3)}d\lambda_2d\lambda_4
\end{equation*}
for $c>0$. Using integrations by parts, $f(c)\to 1$ as $n\to \infty$ as long as $x,y\neq 0, u_1+u_3<t$. Moreover, $f(c)\les (1+|\log c(t-u_1-u_3)|+|\log|x||)(1+|\log c(t-u_1-u_3)|+|\log|y||)$. On the other hand, we note that
\begin{equation*}
f(\frac{1}{2})\leq \frac{1}{(\log n)^2}\int_{\R_+^2}\frac{1}{\lambda_2}e^{-\lambda_2}
\frac{1}{\lambda_4}e^{-\lambda_4}1_{\frac{|x|^2}{2\lambda_2}+\frac{|y|^2}{2\lambda_4}\leq n(t-u_1-u_3)}d\lambda_2d\lambda_4\leq f(1),
\end{equation*}
so by the Lebesgue dominated convergence theorem, we have $\E\{(i)\}\to \frac{t^2}{8\pi^2}\hat{R}_p(0)^2$.

$(ii)$: by a similar change of variables as for $(i)$, we have
\begin{equation*}
\begin{aligned}
\E\{ (ii)\}=\frac{n^2}{a(n)^4}\int_{\R_+^4}\int_{\R^{3d}}&\frac{R_p(x-z)}{|x|^{d-2}}\frac{R_p(y-z)}{|y|^{d-2}}\frac{1}{4\pi^d}\lambda_2^{\frac{d}{2}-2}e^{-\lambda_2}
\lambda_4^{\frac{d}{2}-2}e^{-\lambda_4}\\
&1_{0\leq u_1+u_3\leq t; \frac{|x|^2}{2\lambda_2}+\frac{|y|^2}{2\lambda_4}\leq n(t-u_1-u_3)}q_{nu_3}(z)dxdydzd\lambda_2d\lambda_4du_1du_3.
\end{aligned}
\end{equation*}


By a change of variables and integration by parts in $\lambda_2,\lambda_4$, we have
\begin{equation*}
\begin{aligned}
&|\E\{(ii)\}|\\
 \les& \left(\frac{n}{\log n}\right)^2\int_{[0,t]^2}\int_{\R^6}|R_p(\sqrt{n}(x-z))R_p(\sqrt{n}(y-z))|q_{u_3}(z)\left(e^{-\frac{|x|^2}{2u}}\log \frac{2u}{|x|^2}+\int_{\frac{|x|^2}{2u}}^\infty \log \lambda e^{-\lambda}d\lambda\right)\\
& \left(e^{-\frac{|y|^2}{2u}}\log \frac{2u}{|y|^2}+\int_{\frac{|y|^2}{2u}}^\infty \log \lambda e^{-\lambda}d\lambda\right)dudu_3dxdydz.
\end{aligned}
\end{equation*}
Note that $e^{-\frac{|x|^2}{2u}}\log \frac{2u}{|x|^2}+\int_{\frac{|x|^2}{2u}}^\infty \log \lambda e^{-\lambda}d\lambda\les 1+|\log u|+|\log |x||$. By Lemma \ref{lem:convolLogPotential}, we have
\begin{equation*}
\begin{aligned}
&\frac{n}{\log n}\int_{\R^2}|R_p(\sqrt{n}(x-z))|\left(e^{-\frac{|x|^2}{2u}}\log \frac{2u}{|x|^2}+\int_{\frac{|x|^2}{2u}}^\infty \log \lambda e^{-\lambda}d\lambda\right)dx\\
\les &\frac{1}{\log n}\left(1+|\log u|+|\log |z||+\log n1_{|z|<\frac{2}{\sqrt{n}}}\right).
\end{aligned}
\end{equation*}
The integral in $y$ is controlled in the same way and we obtain
\begin{equation*}
|\E\{(ii)\}|\les \int_{[0,t]^2}\int_{\R^2}\frac{1}{(\log n)^2}\left(1+|\log u|+|\log |z||+\log n1_{|z|<\frac{2}{\sqrt{n}}}\right)^2q_{u_3}(z)dzdudu_3.
\end{equation*}
So $|\E\{(ii)\}|\to 0$ as $n\to \infty$.

$(iii)$: by a similar change of variables as for $(i)$ and by symmetry of $R(x)$, we have
\begin{equation*}
\begin{aligned}
\E\{ (iii)\}=\frac{n^2}{a(n)^4}\int_{\R_+^4}\int_{\R^{3d}}&\frac{R_p(x-y+z)}{|x|^{d-2}}\frac{R_p(y)}{|y|^{d-2}}\frac{1}{4\pi^d}\lambda_2^{\frac{d}{2}-2}e^{-\lambda_2}
\lambda_4^{\frac{d}{2}-2}e^{-\lambda_4}\\
&1_{0\leq u_1+u_3\leq t; \frac{|x|^2}{2\lambda_2}+\frac{|y|^2}{2\lambda_4}\leq n(t-u_1-u_3)}q_{nu_3}(z)dxdydzd\lambda_2d\lambda_4du_1du_3.
\end{aligned}
\end{equation*}


After integrations by parts in $\lambda_2,\lambda_4$, we have
\begin{equation*}
\begin{aligned}
&|\E\{(iii)\}|\\
 \les& \left(\frac{n}{\log n}\right)^2\int_{[0,t]^2}\int_{\R^6}|R_p(\sqrt{n}(x-y+z))R_p(\sqrt{n}y)|q_{u_3}(z)\left(e^{-\frac{|x|^2}{2u}}\log \frac{2u}{|x|^2}+\int_{\frac{|x|^2}{2u}}^\infty \log \lambda e^{-\lambda}d\lambda\right)\\
& \left(e^{-\frac{|y|^2}{2u}}\log \frac{2u}{|y|^2}+\int_{\frac{|y|^2}{2u}}^\infty \log \lambda e^{-\lambda}d\lambda\right)dudu_3dxdydz.
\end{aligned}
\end{equation*}
Note that  $e^{-\frac{|x|^2}{2u}}\log \frac{2u}{|x|^2}+\int_{\frac{|x|^2}{2u}}^\infty \log \lambda e^{-\lambda}d\lambda\les 1+|\log u|+|\log |x||$. By applying Lemma \ref{lem:convolLogPotential} to the integral in $x$, we have
\begin{equation*}
\begin{aligned}
&\frac{n}{\log n}\int_{\R^2}|R_p(\sqrt{n}(x-(y-z)))|\left(e^{-\frac{|x|^2}{2u}}\log \frac{2u}{|x|^2}+\int_{\frac{|x|^2}{2u}}^\infty \log \lambda e^{-\lambda}d\lambda\right)dx\\
\les &\frac{1}{\log n}\left(1+|\log u|+|\log |y-z||+\log n1_{|y-z|<\frac{2}{\sqrt{n}}}\right).
\end{aligned}
\end{equation*}
So
\begin{equation*}
\begin{aligned}
|\E\{(iii)\}| \les
\frac{n}{(\log n)^2}\int_{[0,t]^2}\int_{\R^4}&\left(1+|\log u|+|\log |y-z||+\log n1_{|y-z|<\frac{2}{\sqrt{n}}}\right)|R_p(\sqrt{n}y)|q_{u_3}(z)\\
&(1+|\log u|+|\log |y||)dydzdu_3du.
\end{aligned}
\end{equation*}
Since $|R_p(\sqrt{n}y)|\les 1\wedge |\sqrt{n}y|^{-\alpha}$ for some $\alpha>2$, by Lemma \ref{lem:convolLogPotential}, we know $\E\{(iii)\}\to 0$ as $n\to \infty$.


\medskip

We now consider the case $d\geq 3$.

$(i)$: after Fourier transform and changing of variables $u_i=s_i-s_{i-1}$ for $i=1,2,3,4$ with $s_0=0$, we derive
\begin{equation}
\begin{aligned}
\E\{(i)\}=&\frac{1}{(2\pi)^{2d}n^2}\int_{\R_+^4}\int_{\R^{2d}}1_{\sum_{i=1}^4u_i\leq nt}\hat{R}_p(\xi_1)\hat{R}_p(\xi_2)e^{-\frac{1}{2}|\xi_1|^2u_2}
e^{-\frac{1}{2}|\xi_2|^2u_4}d\xi_1d\xi_2 du\\
=&\frac{1}{(2\pi)^{2d}}\int_{\R_+^2}\int_{\R^{2d}}\hat{R}_p(\xi_1)\hat{R}_p(\xi_2)1_{0\leq u_1+u_3\leq t}F_n(\frac12 |\xi_1|^2,\frac12 |\xi_2|^2,t-u_1-u_3)d\xi_1d\xi_2du,
\end{aligned}
\end{equation}
where $F_n(a,b,t):=\int_{\R_+^2}1_{0\leq s+u\leq nt}e^{-as}e^{-bu}dsdu$ for $a\geq 0, b\geq 0$. It is straightforward to check that $abF_n(a,b,t)$ is uniformly bounded and $F_n(a,b,t)\to \frac{1}{ab}$ as $n\to\infty$. Thus,
\begin{equation}
\begin{aligned}
\E\{(i)\}\to& \frac{1}{(2\pi)^{2d}}\int_{\R_+^2}\int_{\R^{2d}}\hat{R}_p(\xi_1)\hat{R}_p(\xi_2)1_{0\leq u_1+u_3\leq t}\frac{4}{|\xi_1|^2|\xi_2|^2}d\xi_1d\xi_2du_1du_3\\
=&\frac{2t^2}{(2\pi)^{2d}}\left(\int_{\R^d}\frac{\hat{R}_p(\xi)}{|\xi|^2}d\xi\right)^2.
\end{aligned}
\end{equation}

$(ii)$: similarly we have
\begin{equation}
\begin{aligned}
\E\{(ii)\}=&\frac{1}{(2\pi)^{2d}n^2}\int_{\R_+^4}\int_{\R^{2d}}1_{\sum_{i=1}^4u_i\leq nt}\hat{R}_p(\xi_1)\hat{R}_p(\xi_2)
e^{-\frac{1}{2}|\xi_1|^2u_2}e^{-\frac{1}{2}|\xi_1+\xi_2|^2u_3}e^{-\frac{1}{2}|\xi_2|^2u_4}d\xi_1d\xi_2du\\
\les & t\int_{\R^{2d}}\int_0^t\frac{\hat{R}_p(\xi_1)\hat{R}_p(\xi_2)}{|\xi_1|^2|\xi_2|^2}e^{-\frac{1}{2}|\xi_1+\xi_2|^2nu_3}du_3d\xi_1d\xi_2\to 0
\end{aligned}
\end{equation}
as $n\to \infty$.

$(iii)$: by the same change of variables, we obtain
\begin{equation}
\begin{aligned}
\E\{(iii)\}=&\frac{1}{(2\pi)^{2d}n^2}\int_{\R_+^4}\int_{\R^{2d}}1_{\sum_{i=1}^4u_i\leq nt}\hat{R}_p(\xi_1)\hat{R}_p(\xi_2)
e^{-\frac{1}{2}|\xi_1|^2u_2}e^{-\frac{1}{2}|\xi_1+\xi_2|^2u_3}e^{-\frac{1}{2}|\xi_1|^2u_4}d\xi_1d\xi_2du\\
\les & \frac{t}{n}\int_{\R^{2d}}\int_{\R_+^2}1_{u_3+u_4\leq nt}\frac{\hat{R}_p(\xi_1)}{|\xi_1|^2}\hat{R}_p(\xi_2)e^{-\frac{1}{2}|\xi_1+\xi_2|^2u_3}e^{-\frac{1}{2}|\xi_1|^2u_4}du_3du_4
d\xi_1d\xi_2\\
\les &
\frac{t}{n}\int_{\R^{2d}}\int_{\R_+^2}1_{u_3+u_4\leq nt}\frac{\hat{R}_p(\xi_1)\hat{R}_p(\xi_2)}{|\xi_1|^2|\xi_2|^2}(|\xi_1+\xi_2|^2+|\xi_1|^2)
e^{-\frac{1}{2}|\xi_1+\xi_2|^2u_3}e^{-\frac{1}{2}|\xi_1|^2u_4}du_3du_4
d\xi_1d\xi_2\\
\leq &t\int_{\R^{2d}}\int_0^t\int_{\R_+}\frac{\hat{R}_p(\xi_1)\hat{R}_p(\xi_2)}{|\xi_1|^2|\xi_2|^2}|\xi_1+\xi_2|^2
e^{-\frac{1}{2}|\xi_1+\xi_2|^2u_3}e^{-\frac{1}{2}|\xi_1|^2nu_4}du_3du_4
d\xi_1d\xi_2\\
+ &t\int_{\R^{2d}}\int_0^t\int_{\R_+}\frac{\hat{R}_p(\xi_1)\hat{R}_p(\xi_2)}{|\xi_1|^2|\xi_2|^2}|\xi_1|^2
e^{-\frac{1}{2}|\xi_1|^2u_4}e^{-\frac{1}{2}|\xi_1+\xi_2|^2nu_3}du_4du_3
d\xi_1d\xi_2\to 0
\end{aligned}
\end{equation}
as $n\to \infty$.


To summarize, we have shown that $\E\{\V_n^2\}\to \sigma_d^4t^2$. The proof is complete.
\end{proof}

\begin{remark}
The proof of Lemma \ref{lem:con2ndmomentVar} only requires $R(x)$ to be symmetric, bounded, and to satisfy certain integrability condition. In particular, if $R(x)$ is compactly supported, then the result holds. This will be used in the proof of tightness.
\label{remark:tight}
\end{remark}

The following lemma proves that those cross terms appearing in the conditional variance vanish in the limit. When $d\geq 3$, as in the proof of Lemma \ref{lem:con1stmomentVar} and \ref{lem:con2ndmomentVar}, we use the Fourier transform and the integrability of $\hat{R}_p(\xi)|\xi|^{-2}$ so that the proof also applies to the degenerate case.

\begin{lemma}
$\frac{1}{a(n)^2}\int_{nt_{i-1}}^{nt_i}\int_{nt_{j-1}}^{nt_j}R_p(B_s-B_u)dsdu\to 0$ in probability when $i\neq j$.
\label{lem:concrossMoment}
\end{lemma}

\begin{proof}
Assume $i>j$,


Consider the case $d=2$. For fixed $x,u$, by change of variables $\lambda=\frac{|x|^2}{2(s-u)}$, we have
 \begin{equation*}
 \begin{aligned}
 &\frac{1}{a(n)^2}\int_{nt_{i-1}}^{nt_i}\int_{nt_{j-1}}^{nt_j}\E\{|R_p(B_s-B_u)|\}dsdu\\
=&\frac{n}{a(n)^2}\int_{\R^d}\int_{\R_+^2}1_{(\frac{|x|^2}{2n(t_i-u)},\frac{|x|^2}{2n(t_{i-1}-u)})}(\lambda)1_{(t_{j-1},t_j)}(u)\frac{1}{2\pi^{\frac{d}{2}}}
\frac{|R_p(x)|}{|x|^{d-2}}\lambda^{\frac{d}{2}-2}e^{-\lambda}d\lambda dudx.
\end{aligned}
\end{equation*}
Recalling that  $a(n)=\sqrt{n\log n}$, an integration by parts leads to \begin{equation*}
\int_{\R}1_{(\frac{|x|^2}{2n(t_i-u)},\frac{|x|^2}{2n(t_{i-1}-u)})}(\lambda)\lambda^{-1}e^{-\lambda}d\lambda\les 1+\log n+|\log(t_i-u)|+|\log(t_{i-1}-u)|+|\log |x||,
\end{equation*}
and $\frac{1}{\log n}\int_{\R}1_{(\frac{|x|^2}{2n(t_i-u)},\frac{|x|^2}{2n(t_{i-1}-u)})}(\lambda)\lambda^{-1}e^{-\lambda}d\lambda\to 0$ as $n\to \infty$. We apply the dominated convergence theorem to conclude the proof.

Consider now the case $d\geq 3$ and
 \begin{equation}
(i):=\frac{1}{n^2}\int_{\R_+^4}1_{s_1,s_2\in [nt_{i-1},nt_i]}1_{u_1,u_2\in [nt_{j-1},nt_j]}R_p(B_{s_1}-B_{u_1})R_p(B_{s_2}-B_{u_2})dsdu.
\end{equation}
We show $\E\{(i)\}\to 0$ so the cross term goes to zero in probability.
Actually, we have $(i)=2((I)+(II))$, where
\begin{eqnarray}
(I)=\frac{1}{n^2}\int_{\R_+^4}1_{nt_{j-1}\leq u_2\leq u_1\leq nt_j}1_{\leq nt_{i-1}\leq  s_2\leq s_1\leq nt_i}
R_p(B_{s_1}-B_{u_1})R_p(B_{s_2}-B_{u_2})dsdu,\\
(II)=\frac{1}{n^2}\int_{\R_+^4}1_{nt_{j-1}\leq u_1\leq u_2\leq nt_j} 1_{nt_{i-1}\leq  s_2\leq s_1\leq nt_i}
R_p(B_{s_1}-B_{u_1})R_p(B_{s_2}-B_{u_2})dsdu.
\end{eqnarray}
For $(I)$ we have
\begin{equation}
\begin{aligned}
\E\{(I)\}=\frac{1}{(2\pi)^{2d}n^2}\int_{\R_+^4}\int_{\R^{2d}}&1_{nt_{j-1}\leq u_2\leq u_1\leq nt_j}1_{\leq nt_{i-1}\leq  s_2\leq s_1\leq nt_i}
\hat{R}_p(\xi_1)\hat{R}_p(\xi_2)\\
&e^{-\frac{1}{2}|\xi_1|^2(s_1-s_2)}e^{-\frac{1}{2}|\xi_1+\xi_2|^2(s_2-u_1)}e^{-\frac{1}{2}|\xi_2|^2(u_1-u_2)}d\xi_1d\xi_2dsdu,
\end{aligned}
\end{equation}
which implies $\E\{(I)\}\les (t_j-t_{j-1})\int_{\R^{2d}}\int_{t_{i-1}-t_j}^{t_i-t_{j-1}}\frac{\hat{R}_p(\xi_1)\hat{R}_p(\xi_2)}{|\xi_1|^2|\xi_2|^2}e^{-\frac{1}{2}|\xi_1+\xi_2|^2nu}dud\xi_1d\xi_2\to 0$
as $n\to \infty$. Similarly, for $(II)$ we have
\begin{equation}
\begin{aligned}
\E\{(II)\}=\frac{1}{(2\pi)^{2d}n^2}\int_{\R_+^4}\int_{\R^{2d}}&1_{nt_{j-1}\leq u_1\leq u_2\leq nt_j}1_{\leq nt_{i-1}\leq  s_2\leq s_1\leq nt_i}
\hat{R}_p(\xi_1)\hat{R}_p(\xi_2)\\
&e^{-\frac{1}{2}|\xi_1|^2(s_1-s_2)}e^{-\frac{1}{2}|\xi_1+\xi_2|^2(s_2-u_2)}e^{-\frac{1}{2}|\xi_1|^2(u_2-u_1)}d\xi_1d\xi_2dsdu,\\
\end{aligned}
\end{equation}
so \begin{equation}
\begin{aligned}
\E\{(II)\}\les &\frac{1}{n}\int_{[nt_{j-1},nt_i]^2}\int_{\R^{2d}}
\frac{\hat{R}_p(\xi_1)\hat{R}_p(\xi_2)}{|\xi_1|^2|\xi_2|^2}
(|\xi_1+\xi_2|^2+|\xi_1|^2)e^{-\frac{1}{2}|\xi_1+\xi_2|^2u_1}e^{-\frac{1}{2}|\xi_1|^2u_2}d\xi_1d\xi_2du_1du_2\\
\les & (t_i-t_{j-1})\int_{\R^{2d}}\int_{t_{j-1}}^{t_i}
\frac{\hat{R}_p(\xi_1)\hat{R}_p(\xi_2)}{|\xi_1|^2|\xi_2|^2}e^{-\frac{1}{2}|\xi_1+\xi_2|^2nu}dud\xi_1d\xi_2\\
+&(t_i-t_{j-1})\int_{\R^{2d}}\int_{t_{j-1}}^{t_i}
\frac{\hat{R}_p(\xi_1)\hat{R}_p(\xi_2)}{|\xi_1|^2|\xi_2|^2}e^{-\frac12|\xi_1|^2nu}dud\xi_1d\xi_2\to 0
\end{aligned}
\end{equation}
as $n\to\infty$. 
\end{proof}

The two following propositions show the convergence in probability of $\int_{\R^d} \sum_{k=2}^\infty\frac{1}{k!}(i\theta F_n(y))^kdy$.


\begin{proposition}
$\int_{\R^d} F_n(y)^2dy\to \sum_{i=1}^N\alpha_i^2 \sigma_d^2(t_i-t_{i-1})$ in probability.
\label{prop:conVar23d}
\end{proposition}

\begin{proposition}
$\int_{\R^d} \sum_{k=3}^\infty \frac{1}{k!}(i\theta F_n(y))^k dy\to 0$ in probability.
\label{prop:conRemain23d}
\end{proposition}

\begin{proof}[Proof of Proposition \ref{prop:conVar23d}]
Note that
\begin{equation*}
\int_{\R^d}F_n(y)^2dy=\sum_{i,j=1}^N \alpha_i\alpha_j\frac{1}{a(n)^2}\int_{nt_{i-1}}^{nt_i}\int_{nt_{j-1}}^{nt_j}R_p(B_s-B_u)dsdu,
\end{equation*}
when $i=j$, Lemma \ref{lem:con1stmomentVar} and \ref{lem:con2ndmomentVar} lead to
 \begin{equation*}
 \frac{1}{a(n)^2}\int_{nt_{i-1}}^{nt_i}\int_{nt_{i-1}}^{nt_i}R_p(B_s-B_u)dsdu\to
 \sigma_d^2(t_i-t_{i-1})
 \end{equation*}
 in probability as $n\to \infty$.

 When $i\neq j$, by Lemma \ref{lem:concrossMoment}, we have
 \begin{equation}
 \frac{1}{a(n)^2}\int_{nt_{i-1}}^{nt_i}\int_{nt_{j-1}}^{nt_j}R_p(B_s-B_u)dsdu\to 0
 \end{equation}
in probability as $n\to \infty$.
 The proof is complete.
 \end{proof}

\begin{proof}[Proof of Proposition \ref{prop:conRemain23d}]
We will use $C$ for possibly different constants in the following estimation. Recall that $F_n(y)=\sum_{i=1}^N\alpha_i \frac{1}{a(n)}\int_{nt_{i-1}}^{nt_i}\phi(B_s-y)ds$, so we have $|F_n(y)|\leq C\frac{1}{a(n)}\int_0^n |\phi(B_s-y)|ds$, and thus
\begin{equation}
\int_{\R^d}\E\{|F_n(y)|^k\}dy \leq C^k \int_{\R^d}\frac{1}{a(n)^k}\int_{[0,n]^k}\E\{\prod_{i=1}^k |\phi(B_{s_i}-y)|\}dsdy.
\label{eq:eqforRHS}
\end{equation}
From now on, we use $RHS$ to denote the $RHS$ of \eqref{eq:eqforRHS}. By change of variables $u_i=s_i-s_{i-1}$ for $i=1,\ldots,k$ with $s_0=0$, and $\lambda_i=\frac{|x_i|^2}{2u_i}$ for $i=2,\ldots,k$ when $x_i$ is fixed, we have
\begin{equation*}
\begin{aligned}
&RHS\\
=&\frac{C^kk!}{a(n)^k}\int_{\R^{(k+1)d}}\int_{\R_+^k}1_{\sum_{i=1}^k u_i\leq n}|\phi|(y)|\phi|(x_2+y)\ldots |\phi|(\sum_{i=2}^kx_i+y)\prod_{i=1}^k q_{u_i}(x_i)dudxdy\\
=&\frac{C^kk!}{a(n)^k}\int_{\R^{kd}}\int_{\R_+^k}1_{u_1+\sum_{i=2}^k \frac{|x_i|^2}{2\lambda_i}\leq n}|\phi|(y)\frac{|\phi|(x_2+y)}{|x_2|^{d-2}}\ldots \frac{|\phi|(\sum_{i=2}^kx_i+y)}{|x_k|^{d-2}}\prod_{i=2}^k \frac{1}{2\pi^{\frac{d}{2}}}\lambda_i^{\frac{d}{2}-2}e^{-\lambda_i}du_1d\lambda dxdy.
\end{aligned}
\end{equation*}
When $d\geq 3$, note that $\int_{\R^d}|\phi|(y+x)|y|^{2-d}dy$ is uniformly bounded in $x$, so after integration in $x_k,\ldots,x_2,y$ and $\lambda_2,\ldots,\lambda_k$, we have $RHS\leq C^kk!n^{-\frac{k}{2}+1}$ where the factor $n$ comes from the integration in $u_1$. This leads to
\begin{equation*}
\E\{|\int_{\R^d} \sum_{k=3}^\infty \frac{1}{k!}(i\theta F_n(y))^k dy|\}\leq \sum_{k=3}^\infty |C\theta|^k\frac{1}{n^{\frac{k}{2}-1}} \to 0
\end{equation*}
as $n\to \infty$.

When $d=2$, we have
\begin{equation}
\begin{aligned}
RHS\leq C^k\frac{nk!}{a(n)^k}\int_{\R^{kd}}\int_{\R_+^{k-1}}|\phi|(y)|\phi|(x_2+y)\ldots |\phi|(\sum_{i=2}^kx_i+y)\prod_{i=2}^k\frac{1}{2\pi}1_{\lambda_i\geq \frac{|x_i|^2}{2n}} \frac{1}{\lambda_i}e^{-\lambda_i}d\lambda dxdy.
\end{aligned}
\label{eq:d2highOrderterms}
\end{equation}
By integration by parts, we have $\frac{1}{(\log n)^{k-1}}\int_{\R_+^{k-1}}\prod_{i=2}^k\frac{1}{2\pi}1_{\lambda_i\geq \frac{|x_i|^2}{2n}} \frac{1}{\lambda_i}e^{-\lambda_i}d\lambda \les \prod_{i=2}^k (1+|\log |x_i||)$. Since $\phi$ is compactly supported, we know that $x_i, i=2,\ldots,k$ are uniformly bounded. After integration in $x_k,\ldots,x_2,y$, we have $RHS\leq C^kk!\left(\frac{\log n}{n}\right)^{\frac{k}{2}-1}$. So
\begin{equation*}
\E\{|\int_{\R^d} \sum_{k=3}^\infty \frac{1}{k!}(i\theta F_n(y))^k dy|\}\leq \sum_{k=3}^\infty |C\theta|^k\left(\frac{\log n}{n}\right)^{\frac{k}{2}-1} \to 0
\end{equation*}
as $n\to \infty$.
The proof is complete.
\end{proof}

\begin{remark}
In \eqref{eq:d2highOrderterms}, if we choose $a(n)=n^{\frac12}$ instead of $a(n)=(n\log n)^{\frac12}$, by the same calculation we still have
\begin{equation}
\E\{|\int_{\R^d} \sum_{k=3}^\infty \frac{1}{k!}(i\theta F_n(y))^k dy|\}\leq \sum_{k=3}^\infty |C\theta|^k\frac{\log n^{k-1}}{n^{\frac{k}{2}-1}} \to 0,
\end{equation}
and this could be used in the proof for the degenerate Poissonian case when $d=2$.
\label{remark:d2PoissonDEGE}
\end{remark}

Recall \eqref{eq:chrPoisson}, by using Propositions \ref{prop:conVar23d} and \ref{prop:conRemain23d} and the Lebesgue dominated convergence theorem, we have proved
\begin{equation}
\begin{aligned}
\E\{ \exp(i\theta Y_N)\}=\E\{ \exp(\int_{\R^d}\sum_{k=2}^\infty \frac{1}{k!}(i\theta F_n(y))^kdy)\}
\to& \E\{\exp(-\frac{1}{2}\theta^2\sum_{i=1}^N\alpha_i^2\sigma_d^2(t_i-t_{i-1}))\}\\
=&\E\{\exp(i\theta\sigma_d\sum_{i=1}^N\alpha_i(W_{t_i}-W_{t_{i-1}}))\}
\end{aligned}
\end{equation}
when $W_t$ is a standard Brownian motion.

\subsubsection{Gaussian case}

When $d=1$, by Proposition \ref{prop:conVar1d}, we have
\begin{equation}
\begin{aligned}
\E\{ \exp(i\theta Y_N)\}=&\E\{ \exp(-\frac{1}{2}\theta^2\sum_{i,j=1}^N\alpha_i\alpha_j\frac{1}{a(n)^2}\int_{nt_{i-1}}^{nt_i}\int_{nt_{j-1}}^{nt_j}R_g(B_s-B_u)dsdu)\}\\
\to& \E\{\exp(-\frac{1}{2}\theta^2\hat{R}_g(0)\sum_{i,j=1}^N \alpha_i\alpha_j\int_{\R}(L_{t_i}(x)-L_{t_{i-1}}(x))(L_{t_j}(x)-L_{t_{j-1}}(x))dx)\}\\
=&\E\{\exp(i\theta\sigma_d\sum_{i=1}^N\alpha_i(Z_{t_i}-Z_{t_{i-1}}))\},
\end{aligned}
\end{equation}
when $Z_t=\int_{\R}L_t(x)W(dx)$.

When $d\geq 2$, by Proposition \ref{prop:conVar23d}, we have
\begin{equation}
\begin{aligned}
\E\{ \exp(i\theta Y_N)\}=&\E\{ \exp(-\frac{1}{2}\theta^2\sum_{i,j=1}^N\alpha_i\alpha_j\frac{1}{a(n)^2}\int_{nt_{i-1}}^{nt_i}\int_{nt_{j-1}}^{nt_j}R_g(B_s-B_u)dsdu)\}\\
\to& \E\{\exp(-\frac{1}{2}\theta^2\sum_{i=1}^N\alpha_i^2\sigma_d^2(t_i-t_{i-1}))\}
=\E\{\exp(i\theta\sigma_d\sum_{i=1}^N\alpha_i(W_{t_i}-W_{t_{i-1}}))\},
\end{aligned}
\end{equation}
when $W_t$ is a standard Brownian motion.

\subsection{Tightness}

\begin{proposition}
$X_n(t)$ is tight in $\C([0,1])$.
\label{prop:tightNonDEGE}
\end{proposition}

\begin{proof}
Since $X_n(t)=a(n)^{-1}\int_0^{nt}V(B_s)ds$, then $X_n(0)=0$. To prove tightness of $X_n$ by \cite[Theorem 12.3]{billingsley2009convergence}, we only need to show
\begin{equation}
\E\{|X_n(t)-X_n(s)|^\beta\}\leq C|t-s|^{1+\delta}
\label{eq:komoCon}
\end{equation}
for some constant $\beta,C,\delta>0$.

Consider $d=1$. $\E\{|X_n(t)-X_n(s)|^2\}=n^{-\frac{3}{2}}\int_{[0,n(t-s)]^2}\E\{R(B_{u_1}-B_{u_2})\}du_1du_2$. Since $R$ is bounded and compactly supported, we have
    \begin{equation*}
    \begin{aligned}
    \E\{|X_n(t)-X_n(s)|^2\}\leq& \frac{C}{n^{\frac{3}{2}}}\int_0^{n(t-s)}\int_0^{n(t-s)}\Pb(|B_{u_1}-B_{u_2}|\leq C)du_1du_2\\
    =&C\sqrt{n}\int_0^{t-s}\int_0^{t-s}\Pb(|N|\leq \frac{C}{\sqrt{n|u_1-u_2|}})du_1du_2\\
    =&C\int_0^{t-s}\int_0^{t-s}\int_{\R}1_{|x|<\frac{C}{\sqrt{|u_1-u_2|}}}\frac{1}{\sqrt{2\pi}}e^{-\frac{|x|^2}{2n}}dxdu_1du_2\\
    \leq & C\int_0^{t-s}\int_0^{t-s}\frac{1}{\sqrt{|u_1-u_2|}}du_1du_2
    \leq  C(t-s)^{\frac{3}{2}}.
    \end{aligned}
    \end{equation*}

For the case $d\geq 2$, we calculate the $4-$th moment of $X_n(t)-X_n(s)$.
When $V(x)=V_g(x)$ is Gaussian, we have
    \begin{equation*}
    \E\{|X_n(t)-X_n(s)|^4\}=\frac{1}{a(n)^4}\int_{[0,n(t-s)]^4}\sum_{\{\tau_i\}=\{u_i\}}\E\{R_g(B_{\tau_1}-B_{\tau_2})R_g(B_{\tau_3}-B_{\tau_4})\}du.
    \end{equation*}
    When $V(x)=V_p(x)$ is Poissonian, by Lemma \ref{lem:mixing4thmoment}, we have
\begin{equation*}
\E\{|X_n(t)-X_n(s)|^4\}\leq \frac{C}{a(n)^4}\int_{[0,n(t-s)]^4}\sum_{\{\tau_i\}=\{u_i\}}\E\{\varphi^{\frac{1}{2}}(|B_{\tau_1}-B_{\tau_2}|)\varphi^{\frac{1}{2}}(|B_{\tau_3}-B_{\tau_4}|)\}du.
\end{equation*}
where $\varphi$ is a bounded, compactly supported function. The proof of Lemma \ref{lem:con2ndmomentVar} applies to $\varphi^{\frac12}(|x|)$ replacing $R_p(x)$ in light of Remark \ref{remark:tight}.  Since $\E\{\V_n^2\}\leq Ct^2$, in both cases we have $\E\{|X_n(t)-X_n(s)|^4\}\leq C(t-s)^2$.

In \eqref{eq:komoCon}, when $d=1$, we choose $\beta=2,\delta=\frac{1}{2}$ while when $d\geq 2$, we choose $\beta=4,\delta=1$. The proof is complete.
\end{proof}

\section{Degenerate case when $d=1,2$}
\label{sec:4}

Recall that in the degenerate case $d=1,2$, $X_n(t)=n^{-\frac12}\int_0^{nt}V(B_s)ds$, where $V$ is either Gaussian or Poissonian, and we make the key assumption that $\hat{R}(\xi)|\xi|^{-2}$ is integrable. Our goal is to show that $X_n(t)\Rightarrow \sigma W_t$ in $\C([0,1])$ for standard Brownian motion $W_t$ and $\sigma=\sqrt{4(2\pi)^{-d}\int_{\R^d} \hat{R}(\xi)|\xi|^{-2}d\xi}$.

\subsection{Gaussian case}

To consider the finite dimensional distributions, we define $Y_N=\sum_{i=1}^N\alpha_i(X_n(t_i)-X_n(t_{i-1}))$ for $0=t_0<t_1<\ldots<t_N\leq 1$ and $\alpha_i\in \R, i=1,\ldots,N$, so $Y_N$ has mean zero and conditional variance
\begin{equation}
\E\{Y_N^2|B_s, s\in [0,1]\}=\sum_{i,j=1}^N\alpha_i\alpha_j\frac{1}{n}\int_{nt_{i-1}}^{nt_i}\int_{nt_{j-1}}^{nt_j}R(B_s-B_u)dsdu.
\end{equation}
The convergence of $\E\{Y_N^2|B\} \to \sum_{i=1}^N\alpha_i^2\sigma^2(t_i-t_{i-1})$ in probability is given by the proof of the case $d\geq 3$ in Lemmas \ref{lem:con1stmomentVar}, \ref{lem:con2ndmomentVar} and \ref{lem:concrossMoment}.

The proof of tightness in $\C([0,1])$ is the same as in the case $d\geq 3$ in Proposition \ref{prop:tightNonDEGE} so the proof of Gaussian case is complete.

\subsection{Poissonian case}


If we define $F_n(y)=\sum_{i=1}^N\alpha_i n^{-\frac12}\int_{nt_{i-1}}^{nt_i}\phi(B_s-y)ds$ as in the Gaussian case,  we combine Lemmas \ref{lem:con1stmomentVar}, \ref{lem:con2ndmomentVar} and \ref{lem:concrossMoment} to show that
\begin{equation}
\int_{\R^d}F_n(y)^2dy\to \sum_{i=1}^N\alpha_i^2\sigma^2(t_i-t_{i-1}).
\end{equation}
To prove the convergence of the finite dimensional distributions, it suffices to show \begin{equation}
\int_{\R^d}\sum_{k=3}^\infty\frac{1}{k!}(i\theta F_n(y))^kdy\to 0
\label{eq:highOrderDEGE}
\end{equation}
in probability. However, it turns out that a direct proof of \eqref{eq:highOrderDEGE} also involves a tightness result. Instead, we apply Kipnis-Varadhan's approach involving solving a corrector equation and a martingale approximation. It turns out the results in \cite{kipnis1986central} already contain our special case. We briefly recall their results and prove the required assumption holds in our context.

The following Proposition comes from \cite[Theorem 1.8, Corollary 1.9]{kipnis1986central}.
\begin{proposition}
Let $y(t)$ be a Markov process, reversible with respect to a probability measure $\pi$, and let us suppose that the reversible stationary process $\Pb$ with $\pi$ as invariant measure is ergodic. Let $\V$ be a function on the state space in $L^2(\pi)$ satisfying $\int_\Omega \V d\pi=0$ and the condition $\langle -L^{-1}\V,\V\rangle<\infty$ with $\langle.,.\rangle$ denoting the inner product in $L^2(\pi)$ and $L$ the infinitesimal generator of the process. Let $X(t)=\int_0^t V(y(s))ds$, then $\frac{1}{\sqrt{n}} X(nt)$ satisfies a functional central limit theorem relative to $\Pb$ and the limiting variance $\sigma^2=2\langle -L^{-1}\V,\V\rangle$.
\label{prop:resultKV}
\end{proposition}

In the following, we present a setup of Brownian motion in random scenery borrowed from \cite[Section 9.3]{komorowski2012fluctuations}, to which Proposition \ref{prop:resultKV} can be applied.

Let $(\Omega,\mathcal{F},\pi)$ be a probability space associated with a group of measure-preserving transformations $\{\tau_x,x\in \R^d\}$, i.e, $\pi\circ \tau_x=\pi$ for all $x\in \R^d$. Furthermore, its action is ergodic and stochastically continuous, i.e.,
\begin{enumerate}
\item $\pi\{A\}=0$ or $1$ for any event $A$ such that $\pi\{A \bigtriangleup\tau_x(A)\}=0$ for all $x\in \R^d$ and
\item for any $\delta>0$ and $G$ bounded we have
\begin{equation}
\lim_{h\to 0} \pi\{\omega:|G(\tau_h\omega)-G(\omega)|\geq \delta\}=0.
\end{equation}
\end{enumerate}

The probability space $(\Omega,\mathcal{F},\pi)$ satisfying the above assumption is called random medium.

For any $f\in L^2(\pi)$, let $T_xf(\omega)=f(\tau_x\omega)$. The family $\{T_x, x\in\R^d\}$ forms a $d-$parameter group of unitary operators on $L^2(\pi)$, and stochastic continuity implies that the group is strongly continuous. The generators of the group $\{T_x,x\in \R^d\}$ correspond to differentiation (in $L^2(\pi)$) in the canonical directions $e_k$ and are denoted by $\{D_k,k=1,\ldots,d\}$.

Since $\{T_x,x\in\R^d\}$ is strongly continuous, by spectral theory we have
\begin{equation}
T_x=\int_{\R^d}e^{i\xi x}U(d\xi)
\end{equation}
with  $U(d\xi)$ the associated projection valued measure. Let $L=\frac12\sum_{k=1}^d D_k^2$ and $\V\in L^2(\pi)$ satisfies $\int_\Omega \V d\pi=0$ and $\langle -L^{-1}\V,\V\rangle<\infty$.  By the spectral representation, we have 
\begin{equation}
-L^{-1}=2\int_{\R^d}|\xi|^{-2}U(d\xi).
\end{equation}
 Let $\hat{R}_\V(\xi)$ be the power spectrum associated with $\V$, i.e., \begin{equation}
 \hat{R}_\V(\xi)d\xi=(2\pi)^d\langle U(d\xi)\V,\V\rangle,
 \end{equation} and we obtain that
\begin{equation}
\langle -L^{-1}\V,\V\rangle=\langle 2\int_{\R^d}|\xi|^{-2}U(d\xi)\V,\V\rangle=\frac{2}{(2\pi)^d}\int_{\R^d}\frac{\hat{R}_\V(\xi)}{|\xi|^2}d\xi.
\label{eq:LVV}
\end{equation}
Therefore, the condition that $\langle-L^{-1}\V,\V\rangle<\infty$ is equivalent with the integrability of $\hat{R}_\V(\xi)|\xi|^{-2}$. On the other hand, by defining $V(x)=\V(\tau_{-x}\omega)$, we obtain that
\begin{equation}
R(x)=\E\{V(-x)V(0)\}=\langle T_x\V,\V\rangle=\int_{\R^d}e^{i\xi  x}\langle U(d\xi)\V,\V\rangle=\frac{1}{(2\pi)^d}\int_{\R^d}e^{i\xi  x}\hat{R}_\V(\xi)d\xi,
\end{equation} 
so $\hat{R}(\xi)=\hat{R}_\V(\xi)$.

%
%
%

Now if we consider a stationary ergodic random scenery $V(x)=\V(\tau_{-x}\omega)$, the Brownian motion in random scenery with property scaling is $X_n(t)=\frac{1}{\sqrt{n}}\int_0^{nt}V(B_s)ds=\frac{1}{\sqrt{n}}\int_0^{nt}\V(y_s^\omega)ds$ with 
\begin{equation}
y_s^\omega:=\tau_{-B_s}\omega,
\end{equation}
and we only have to prove the environment process $y_s^\omega$ satisfies the assumptions in Proposition \ref{prop:resultKV} and the Poissonian random potential lives indeed on a random medium. In the following, we denote the probability only with respect to Brownian motion by $\Pb_B$.

\begin{lemma}
$y_s^\omega$ is a stationary, ergodic Markov process taking values in $\Omega$, reversible with respect to the invariant measure $\pi$.
\end{lemma}

\begin{proof}
Since $\tau_x$ is a group of transformations, we have $\tau_{-B_t}\omega=\tau_{-(B_t-B_s)}\tau_{-B_s}\omega$, and by the independence of increments of Brownian motion, $y_s^\omega$ is Markov.

Now we show $y_s^\omega$ is reversible with respect to $\pi$ by proving
\begin{equation}
\int_{A_1}\pi(d\omega)\Pb_B(y_t^\omega\in A_2)=\int_{ A_2}\pi(d\omega)\Pb_B(y_t^\omega\in A_1)
\label{eq:reversible}
\end{equation}
for any $A_1,A_2\in \mathcal{F}$. Actually, we have \begin{eqnarray}
\int_{A_1}\pi(d\omega)\Pb_B(y_t^\omega\in A_2)&=&
\int_{\Omega}\int_{\R^d}1_{\omega\in A_1}1_{\tau_{-x} \omega\in A_2}q_t(x)dx\pi(d\omega),\\
\int_{ A_2}\pi(d\omega)\Pb_B(y_t^\omega\in A_1)&=&
\int_{\Omega}\int_{\R^d}1_{\omega\in A_2}1_{\tau_{-x} \omega \in A_1}q_t(x)dx\pi(d\omega).
\end{eqnarray}
Using measure-preserving property of $\tau_x$ and the fact that $q_t(x)=q_t(-x)$, \eqref{eq:reversible} is proved.

Since $y_s^\omega$ is reversible with respect to $\pi$, $\pi$ is an invariant measure. Furthermore, $y_s^\omega$ starts from its invariant measure, so it is stationary.



For ergodicity, we only need to show that if $A\in \mathcal{F}$ such that $\Pb_B(y_s^\omega\in A)=1_{\omega\in A}$ for all $s\geq 0$, then $\pi(A)=0$ or $1$. Actually, $\Pb_B(y_s^\omega\in A)=\int_{\R^d}1_{\tau_{-x}\omega\in A}q_s(x)dx=1_{\omega\in A}$ implies $1_{\tau_{-x} \omega\in A}=1_{\omega\in A}$ for all $x\in \R^d$, since $q_s(x)>0,\forall x\in \R^d$. By the ergodicity of $\tau_x$, we have $\pi(A)=0$ or $1$.
\end{proof}

The infinitesimal generator of $y_s^\omega$ is given by $L$. For detailed proof, we refer to \cite[Proposition 9.8]{komorowski2012fluctuations}.

Next, we show that the Poissonian potential fits the framework of random medium.

Let $\omega=\omega(dy)$ be a Poissonian field in $\R^d$ with Lebesgue measure $dy$ as its intensity, we can write it as
\begin{equation}
\omega(dy)=\sum_i \delta_{\xi_i}(dy)
\end{equation}
where $\delta_z(dy)$ is the Dirac delta measure at $z$, $\{\xi_i\}$ is the Poisson point process with $\pi$ being its law. If $\omega(A)$ denotes the number of points in $\{\xi_i\}$ that fall inside $A$, we have $\pi(\omega(A)=n)=e^{-|A|}|A|^n(n!)^{-1}$ with $|A|$ the Lebesgue measure of $A$. 

The group of transformation $\{\tau_x,x\in \R^d\}$ acts on $\omega=\omega(dy)\in \Omega$ as
\begin{equation}
(\tau_x\omega)(dy)=\sum_i \delta_{x+\xi_i}(dy),
\end{equation}
and we have the following standard result:
\begin{lemma}
$\{\tau_x,x\in\R^d\}$ is measure-preserving, ergodic and stochastically continuous in the following sense:
\begin{enumerate}
\item $\pi\circ \tau_x=\pi$ for all $x\in \R^d$, where $\pi\circ\tau_x$ is the law of Poisson point process $\{\xi_i+x\}$.
\item any $f\in L^2(\pi)$ that satisfies $f(\tau_x\omega)=f(\omega)$ for all $x\in \R^d$ has to be a contant.
\item for any $\delta>0$ and $G$ bounded, we have $\lim_{h\to 0}\pi(\omega:|G(\tau_h \omega)-G(\omega)|\geq \delta)=0$.
\end{enumerate}
\end{lemma}

\begin{proof}
For the measure-preserving property, since the Laplace functional characterize the Poisson point process, for any positive smooth test function of compact support $f$, we consider
\begin{equation}
\begin{aligned}
\E\{e^{-\int_{\R^d}f(y)(\tau_x\omega)(dy)}\}=\exp(\int_{\R^d}(e^{-f(x+y)}-1)dy)=&\exp(\int_{\R^d}(e^{-f(y)}-1)dy)\\
=&\E\{e^{-\int_{\R^d}f(y)\omega(dy)}\},
\end{aligned}
\end{equation}
so $\pi\circ \tau_x=\pi$.

For ergodicity, if $A\cap B=\emptyset$, $\omega(A)$ and $\omega(B)$ are independent, so mixing property implies ergodicity. 

 
For stochastic continuity, by approximation, we can assume that $$G(\omega)=\mathcal{G}(\int_{\R^d}\phi_1(x)\omega(dx),\ldots,\int_{\R^d}\phi_N(x)\omega(dx))$$ for some test functions $\G,\phi_i$, hence we only need to show that $\int_{\R^d}\phi_i(x)(\tau_h\omega)(dx)\to \int_{\R^d}\phi_i(x)\omega(dx)$ in $L^2(\pi)$, and this comes from the fact that
\begin{equation}
\E\{\left(\int_{\R^d}\phi_i(x)(\tau_h\omega)(dx)-\int_{\R^d}\phi_i(x)\omega(dx)\right)^2\}=\int_{\R^d}(\phi_i(x+h)-\phi_i(x))^2dx\to 0
\end{equation}
as $h\to 0$.
\end{proof}

To summarize, we could apply Proposition \ref{prop:resultKV} to $\V(\omega)=\int_{\R^d}\phi(-y)\omega(dy)$, which leads to $V(x)=\V(\tau_{-x}\omega)=\int_{\R^d}\phi(x-y)\omega(dy)$. We only need to recall \eqref{eq:LVV} that \begin{equation}
\sigma^2=2\langle -L^{-1}\V,\V\rangle=\frac{4}{(2\pi)^d}\int_{\R^d}\frac{\hat{R}_\V(\xi)}{|\xi|^{2}}d\xi=\frac{4}{(2\pi)^d}\int_{\R^d}\frac{\hat{R}(\xi)}{|\xi|^{2}}d\xi
\end{equation}
to complete the proof.

\begin{remark}
We point out that by martingale approximation, the results obtained in \cite{kipnis1986central} is stronger than annealed convergence. It is weak convergence in measure, see \cite[Remark (1.10)]{kipnis1986central}.
\end{remark}

\begin{remark}
When $d=2$, by Remark \ref{remark:d2PoissonDEGE}, we could derive  \eqref{eq:highOrderDEGE} and prove the convergence of the finite dimensional distributions in the Poissonian case by the method of characteristic functions. When $d=1$, the estimation turns out to be more involved and our method does not lead to \eqref{eq:highOrderDEGE}. In both cases, the proof of tightness as in Proposition \ref{prop:tightNonDEGE} fails to hold. To use the same fourth moment method, for technical reasons we need the more restrictive condition that $|\hat{\phi}(\xi)||\xi|^{-2}=\sqrt{\hat{R}(\xi)}|\xi|^{-2}$ is integrable.
\end{remark}

\section{Conclusions and discussions}
\label{sec:5}
In this paper, we have proved an invariance principle for Brownian motion in a Gaussian or Poissonian scenery in all dimension. The result is consistent with the discrete case \cite{kesten1979limit, bolthausen1989central} and other types of potentials in the continuous case \cite{remillard1991limit, kipnis1986central}. Our main contribution is the non-degenerate case $d=2$, where a logarithm scaling shows up and the functional central limit theorem for martingale can not be applied as in \cite{kipnis1986central}. It is natural to expect the invariance principle to hold as long as the random scenery is sufficiently mixing, e.g. in our case, the covariance function $R(x)$ is compactly supported. In the non-degenerate case, when $d=1$, the limit is of the form $\int_{\R}L_t(x)W(dx)$ for Brownian local time $L_t(x)$ and independent white noise $W(dx)$, and when $d\geq 2$, the limit is Brownian motion. However, as observed in \cite{gu2013weak}, when the random scenery is long-range correlated, such convergence does not hold and depending on the tail of covariance function, we need to choose different scaling factors.

In the degenerate case, i.e., $\hat{R}(0)=0$ with $\hat{R}(\xi)|\xi|^{-2}$ integrable when $d=1,2$, we have derived the limits with scaling factor $n^{-\frac12}$. The results are essentially the same as in $d\geq 3$ and all directly come from \cite{kipnis1986central}, since under the general assumption of stationarity and ergodicity, the only requirement for their result to hold is the finiteness of asymptotic variance, i.e., integrability of $\hat{R}(\xi)|\xi|^{-2}$. Brownian motion turns out to be the limit for $d=1$ as well.

\section*{Acknowledgment}  We would like to thank the anonymous referee for pointing out several possible improvements in the original manuscript. This paper was partially funded by AFOSR Grant NSSEFF- FA9550-10-1-0194 and NSF grant DMS-1108608.

\appendix
\section{Technical lemmas}

\begin{lemma}
Assume $d=2, \alpha>2, c>0$. Then we have the following inequalities
\begin{eqnarray}
\int_{\R^2}\left( 1\wedge \frac{1}{|\sqrt{n}y|^{\alpha}}\right)|\log |y||dy &\les& \frac1n +\frac1n \log n,\\
\int_{\R^2}
\left( 1\wedge \frac{1}{|\sqrt{n}(x-y)|^{\alpha}}\right)|\log|y||dy &\les&
\frac{1}{n}+\frac{1}{n}|\log|x||+\frac{1}{n}\log n 1_{|x|<\frac{2}{\sqrt{n}}},\\
\int_{|x-y|<\frac{c}{\sqrt{n}}}|\log|y||dy&\les& \frac1n+\frac1n |\log|x||+\frac1n \log n 1_{|x|<\frac{2c}{\sqrt{n}}},\\
\int_{|x-y|<\frac{c}{\sqrt{n}}}\left( 1\wedge \frac{1}{|\sqrt{n}y|^{\alpha}}\right) |\log |y||dy
&\les& \frac1n+\frac1n|\log |x||+\frac1n \log n1_{|x|<\frac{2c}{\sqrt{n}}},
\end{eqnarray}
\begin{equation}
\begin{aligned}
&\int_{\R^2} |\log  |x-y||\left( 1\wedge \frac{1}{|\sqrt{n}y|^{\alpha}}\right)|\log |y||dy\\
\les &\frac{1}{n}+\frac{1}{n}\log n+\frac{1}{n}|\log |x||+\frac{1}{n}\log n|\log |x||+\frac{1}{n} (\log |x|)^2+\frac{1}{n}(\log n)^2 1_{|x|<\frac{1}{\sqrt{n}}}.
\end{aligned}
\end{equation}
\label{lem:convolLogPotential}
\end{lemma}

\begin{proof}
For the first inequality, we have
\begin{equation*}
\int_{\R^2}\left( 1\wedge \frac{1}{|\sqrt{n}y|^{\alpha}}\right)|\log |y||dy
=
-\int_0^{\frac{1}{\sqrt{n}}}r\log rdr+\frac{1}{n^{\frac{\alpha}{2}}}\int_{\frac{1}{\sqrt{n}}}^\infty \frac{1}{r^{\alpha-1}}|\log r|dr,
\end{equation*}
and by integrations by parts, we have $-\int_0^{\frac{1}{\sqrt{n}}}r\log rdr\les \frac{1}{n}\log n$ and $
\frac{1}{n^{\frac{\alpha}{2}}}\int_{\frac{1}{\sqrt{n}}}^\infty \frac{1}{r^{\alpha-1}}|\log r|dr\les \frac{1}{n}(1+\log n)$.

For the other inequalities, they are all in the form of convolutions and are proved in a similar way. We only present the proof for the second one. We have
\begin{equation*}
\int_{\R^2}
\left( 1\wedge \frac{1}{|\sqrt{n}(x-y)|^{\alpha}}\right)|\log|y||dy=(I)+(II),
\end{equation*}
where
\begin{eqnarray*}
(I)&=&\int_{|x-y|<\frac{1}{\sqrt{n}}}|\log |y||dy,\\
(II)&=& \frac{1}{n^{\frac{\alpha}{2}}}\int_{|x-y|>\frac{1}{\sqrt{n}}}\frac{1}{|x-y|^{\alpha}}|\log |y||dy.
\end{eqnarray*}

Let $\rho=|x|$, and define $B(x,r)=\{y:|x-y|\leq r\}$, $(i)=\{y: |y|<|y-x|\}, (ii)=\{y:|y|\geq |y-x|\}$. We divide $\R^d$ into three disjoint parts, $A_1=B(0,\rho)\bigcap (i)$, $A_2=B(x,\rho)\bigcap (ii)$, and $A_3=\R^d\backslash (A_1\bigcup A_2)$.

For $(I)$, when $y\in A_1$, $|y-x|\geq \frac{\rho}{2}$, so $\rho\leq\frac{2}{\sqrt{n}}$ and $\int_{A_1}|\log |y||dy\leq \int_0^\rho|\log r|rdr=\rho^2(\frac{1}{4}-\frac{1}{2}\log \rho)$, so we have $\int_{A_1}|\log |y||dy\les \frac{1}{n}(1+\log n)1_{\rho\leq\frac{2}{\sqrt{n}}}$. When $y\in A_2$, $2\rho\geq|y|\geq\frac{\rho}{2}$, so $|\log |y||\les 1+|\log \rho|$, thus $\int_{A_2}|\log |y||dy\les \frac{1}{n}(1+|\log \rho|)$. When $y\in A_3$, $\rho\leq |y|\leq 2|y-x|\leq \frac{2}{\sqrt{n}}$, so $\int_{A_3}|\log |y||dy\leq \int_\rho^{\frac{2}{\sqrt{n}}}r|\log r|dr\les \frac{1}{n}(1+\log n1_{\rho\leq \frac{2}{\sqrt{n}}})$. Therefore, we have shown that \begin{equation*}
(I)\les \frac{1}{n}(1+|\log \rho|+\log n1_{\rho\leq \frac{2}{\sqrt{n}}}).
\end{equation*}

For $(II)$, by a similar discussion, when $y\in A_1$, $|x-y|\geq \frac{\rho}{2}$, so if $\rho>1$, $\frac{1}{n^\frac{\alpha}{2}}\int_{A_1}\frac{1}{|x-y|^{\alpha}}|\log |y||dy\les \frac{1}{n^{\frac{\alpha}{2}}\rho^\alpha}\int_0^\rho r|\log r|dr\les \frac{1}{n}$, else if $\rho\in (\frac{2}{\sqrt{n}},1]$, we have $\frac{1}{n^\frac{\alpha}{2}}\int_{A_1}\frac{1}{|x-y|^{\alpha}}|\log |y||dy\les \frac{1}{n}(1+|\log \rho|)$, and for the last case $\rho\leq \frac{2}{\sqrt{n}}$, we have $\frac{1}{n^\frac{\alpha}{2}}\int_{A_1}\frac{1}{|x-y|^{\alpha}}|\log |y||dy\les \int_{|y|<\rho}|\log |y||dy\les \frac{1}{n}(1+\log n1_{\rho\leq \frac{2}{\sqrt{n}}})$. When $y\in A_2$, $|\log |y||\les 1+|\log \rho|$, so $\frac{1}{n^\frac{\alpha}{2}}\int_{A_2}
\frac{1}{|x-y|^{\alpha}}|\log |y||dy\les \frac{1}{n}(1+|\log \rho|)$. When $y\in A_3$, $|\log |y||\les 1+|\log |x-y||$, so we only need to estimate $\frac{1}{n^\frac{\alpha}{2}}\int_{\rho\vee \frac{1}{\sqrt{n}}}^\infty\frac{1}{r^{\alpha-1}}(1+|\log r|)dr$. Following the same discussion as in $A_1$, and considering the different cases $\rho>1$, $1\geq \rho>\frac{1}{\sqrt{n}}$ and $\frac{1}{\sqrt{n}}\geq \rho$, we can show $\frac{1}{n^\frac{\alpha}{2}}\int_{\rho\vee \frac{1}{\sqrt{n}}}^\infty\frac{1}{r^{\alpha-1}}(1+|\log r|)dr\les \frac{1}{n}(1+|\log \rho|+\log n1_{\rho\leq \frac{1}{\sqrt{n}}})$. Therefore, we have obtained that
\begin{equation*}
(II)\les \frac{1}{n}(1+|\log \rho|+\log n1_{\rho\leq \frac{2}{\sqrt{n}}}).
\end{equation*}
The proof is complete.
\end{proof}

\begin{lemma}
Let $V(x)$ be a mean zero stationary random field with $\E\{V(x)^6\}<\infty$ satisfying the mixing property \eqref{eq:mixingCondition} with positive, non-increasing mixing coefficient $\varphi$. Then we have
\begin{equation}
|\E\{V(x_1)V(x_2)V(x_3)V(x_4)\}|\leq C\sum_{\{y_k\}=\{x_k\}}\varphi^{\frac{1}{2}}(|y_1-y_2|)\varphi^{\frac{1}{2}}(|y_3-y_4|)\E\{V(x)^6\}^{\frac{2}{3}}.
\label{eq:4thmomentEstimate}
\end{equation}
Thus, \eqref{eq:4thmomentEstimate} holds for the Poissonian potential $V(x)=\int_{\R^d}\phi(x-y)\omega(dy)-c_p$ when $\phi$ is continuous and compactly supported, and the mixing coefficient $\varphi$ could be chosen as some continuous, compactly supported function as well.
\label{lem:mixing4thmoment}
\end{lemma}

\begin{proof}
Let $y_1$ and $y_2$ be two points in $\{x_k\}_{1\leq k\leq 4}$ such that $d(y_1,y_2)\geq d(x_i,x_j)$ for all $1\leq i,j\leq 4$ and such that $d(y_1,\{y_3,y_4\})\leq d(y_2,\{y_3,y_4\})$, where $\{y_k\}_{1\leq k\leq 4}=\{x_k\}_{1\leq k\leq 4}$. We assume $d(y_3,y_1)\leq d(y_4,y_1)$. Therefore by \eqref{eq:mixingCondition}, we have
\begin{equation*}
\mathcal{E}:=|\E\{V(x_1)V(x_2)V(x_3)V(x_4)\}|\les \varphi(2|y_1-y_3|)(\E\{V(y_1)^2\})^{\frac{1}{2}}(\E\{(V(y_2)V(y_3)V(y_4))^2\})^{\frac{1}{2}}.
\end{equation*}
The last two terms are bounded by $\E\{V(x)^6\}^{\frac{1}{6}}$ and $\E\{V(x)^6\}^{\frac{1}{2}}$ respectively. Because $\varphi(r)$ is decaying in $(0,\infty)$, we have $\mathcal{E}\les \varphi(|y_1-y_3|)\E\{V(x)^6\}^{\frac{2}{3}}$.  On the other hand, if $y_4$ is (one of) the closest point(s) to $y_2$, the same argument shows that $\mathcal{E}\les \varphi(|y_2-y_4|)\E\{V(x)^6\}^{\frac{2}{3}}$. Otherwise, $y_3$ is the closest point to $y_2$, and we find $\mathcal{E}\les \varphi(2|y_2-y_3|)\E\{V(x)^6\}^{\frac{2}{3}}$. However, by construction, we have
\begin{equation*}
|y_2-y_4|\leq |y_1-y_2|\leq |y_1-y_3|+|y_2-y_3|\leq 2|y_2-y_3|,
\end{equation*}
so we still have $\mathcal{E}\les \varphi(|y_2-y_4|)\E\{V(x)^6\}^{\frac{2}{3}}$. To summarize, we have
\begin{equation*}
\mathcal{E}\les \varphi^{\frac{1}{2}}(|y_1-y_3|)\varphi^{\frac{1}{2}}(|y_2-y_4|)\E\{V(x)^{6}\}^{\frac{2}{3}},
\end{equation*}
and this completes the proof.
\end{proof}


\end{document}